\documentclass[11pt,reqno]{amsart}
\usepackage[utf8]{inputenc}  
\usepackage[T1]{fontenc}  
\usepackage{pdfsync}
\usepackage{hyperref}
\usepackage{amscd}
\usepackage{amsmath,amsthm,amssymb,amsfonts}
\usepackage{mathrsfs}
\usepackage{graphicx}
\usepackage{fancyhdr}
\usepackage{stmaryrd}
\usepackage[colorinlistoftodos,prependcaption,textsize=tiny]{todonotes}
\usepackage{color}
\usepackage{amsmath}
\usepackage{amsfonts}
\usepackage{amssymb}
\usepackage{geometry}
\usepackage{esint}
\usepackage{mathtools}
\usepackage{chngcntr}
\usepackage{apptools,verbatim}
\AtAppendix{\counterwithin{lemma}{section}}
\numberwithin{equation}{section}

\newtheorem{theorem}{Theorem}[section]
\newtheorem{proposition}{Proposition}[section]
\newtheorem{lemma}{Lemma}[section]

\newtheorem{corollary}{Corollary}[section]
\newtheorem{definition}{Definition}[section]

\theoremstyle{remark}
\newtheorem{remark}{Remark}[section]

\def\R{\mathbb{R}}
\def\C{\mathcal{C}}

\def \H {\mathcal{H}}
\def \p{\partial}

\def \e {\varepsilon}
\def\v{\varepsilon}
\def\S{\mathbb{S}}
\def\n{\nabla}
\def \I {\mathcal I}
\def \F {\mathcal F}

\def \a {\alpha}

\def \O {\Omega}

\def \M {\mathcal{M}}

\def \N {\mathcal{N}}

\def \pin {\text{pin}}
\def\o{\omega}

\DeclareMathOperator{\dist}{dist}
\DeclareMathOperator{\dive}{div}

\DeclareMathOperator{\tr}{tr}
\DeclareMathOperator{\curl}{curl}
\DeclarePairedDelimiter{\abs}{\lvert}{\rvert}

\makeatletter
\def\@seccntformat#1{\@ifundefined{#1@cntformat}%
   {\csname the#1\endcsname\quad}
   {\csname #1@cntformat\endcsname}
}
\makeatother

\author{Micka\"el Dos Santos, Rémy Rodiac, Etienne Sandier}
\title[rapidly oscillating GL energy]{The Ginzburg-Landau energy with a pinning term oscillating faster than the coherence length}

\address[M. Dos Santos]{Laboratoire de Math\'ematiques Blaise Pascal, Universit\'e Clermont Auvergne, UMR CNRS 6620, Campus des
Cézeaux, 63177 Aubière, France}
\email{mickael.dos$\_$santos@uca.fr}

\address[R. Rodiac]{Universit\'e Paris-Saclay, CNRS,  Laboratoire de math\'ematiques d'Orsay, 91405, Orsay, France}
\email{remy.rodiac@universite-paris-saclay.fr}

\address[E.Sandier]{Universit\'e Paris-Est, LAMA - CNRS UMR 8050, 61, Avenue du General de Gaulle, 94010 Cr\'eteil, France}
\email{sandier@u-pec.fr}
\date{\today}

\begin{document}
\begin{abstract}
The aim of this article is to study the magnetic Ginzburg-Landau functional with an oscillating pinning term. We consider here oscillations of the pinning term that are much faster than the coherence length \(\e>0\) which is also the inverse of the Ginzburg-Landau parameter. We study both the case of a periodic potential and of a random stationary ergodic one. We prove that we can reduce the study of the problem to the case where the pinning term is replaced by its average, in the periodic case, and by its expectation with respect to the random parameter in the random case. In order to do that we use a decoupling of the energy (see \cite{Lassoued_Mironescu_1999}) that leads us to the study of the convergence of a scalar positive minimizer of the Ginzburg-Landau energy with pinning term and with homogeneous Neumann boundary conditions. We prove uniform convergence of this minimizer towards the mean value of the pinning term by using a blow-up argument and a Liouville type result for non-vanishing entire solutions of the real Ginzburg-Landau/Allen-Cahn equation, due to Farina \cite{Farina_2003}.
\end{abstract}

\maketitle

\section{Introduction}
 Let \(G \subset \R^2\) be a smooth bounded domain. The main goal of this article is to study the following pinned Ginzburg-Landau energy:

\begin{equation}\label{eq:GL_pinned}
GL_\e^{\text{pin}}(u,A)=\frac12 \int_G |\nabla u-iAu|^2+\frac{1}{4\e^2}\int_G (a_\e(x)-|u|^2)^2 +\frac12 \int_G |\curl A-h_{ex}|^2,
\end{equation}
where \(\e>0, h_{ex}\geq0\) are parameters (here $\v$ is a small parameter: $\v\to 0$), \(u \in H^1(G,\mathbb{C}), A\in H^1(G,\R^2)\), \(\curl A= \p_1A_2-\p_2A_1\) and \(a_\e\) is a function oscillating at a rate \(\delta=\delta_\e{\ll}\e\). More precisely we will study the case where \(a_\e(x)=a_0\left(\frac{x}{\delta} \right)\), with \(a_0: \R^2\rightarrow \R\) a \(1\)-periodic function  and the case where \(a_\e(x)=a_1\left( T(\frac{x}{\delta})\omega\right)\), where \(a_1:\Omega \rightarrow \R\)  is a random variable defined on a probability space \(\Omega\) and \(T\) denotes an action of \(\R^2\) on \(\Omega\) which is stationary and ergodic. Although we do not specify it in our notation the parameter \(\delta\) depends on \(\e\), hence the notation \(a_\e\).
The functional \(GL_\e^{\text{pin}}\) is used to describe the behaviour of type-II superconductors in presence of impurities. In this model \(u\) is a complex order parameter, with \(|u|^2\) representing a normalized density of Cooper pairs of electrons in the sample \(G\), and \(h:=\curl A\) represents the magnetic field inside the sample. When the sample is a homogeneous material, {\it i.e.}\ \(a_\e \equiv 1\), and in the absence of magnetic field, {\it i.e.} when \(A=h_{ex}=0\), the functional \eqref{eq:GL_pinned} has been studied in the pioneering work of Bethuel-Brezis-Hélein \cite{Bethuel_Brezis_Helein_1994}. For the study of the functional with magnetic field we refer to \cite{Sandier_Serfaty_2007} and references therein. In order to describe heterogeneous materials, various authors have considered a modified Ginzburg-Landau energy where various fixed weights appear \cite{Beaulieu_Hadiji_1995,Andre_Shafrir_1998,Lassoued_Mironescu_1999,Andre_Bauman_Phillips_2003}. Oscillating pinning terms were also studied in \cite{Aftalion_Sandier_Serfaty_2001,DMM_2011,DosSantos_2013}. Here our setting is close to the one in \cite{Aftalion_Sandier_Serfaty_2001} except that the assumptions on \(a_\e\) are different: In \cite{Aftalion_Sandier_Serfaty_2001} the pinning term \(a_\e\) oscillates slower than \(\e\) (with our notation, the assumption made in \cite{Aftalion_Sandier_Serfaty_2001} would correspond to \(\delta {\gg}|\log\e|^{-1}\)). The study of \eqref{eq:GL_pinned} combines the difficulties of concentration phenomena in phase transitions theory and of homogenization effects due to oscillations. This is also the case in the recent paper \cite{Alicandro_Braides_Cicalese_DeLuca_Piatniski_2020} where the authors study the homogenization of an oscillating Ginzburg-Landau energy where the oscillating term occurs in the gradient. Oscillations in phase transitions problems were also studied in the context of the Allen-Cahn/ Modica-Mortola functional. On this subject we refer to \cite{Ansini_Braides_Chiado_2003,Dirr_Lucia_Novaga_2006,Dirr_Lucia_Novaga_2008, Cristoferi_Fonseca_Hagerty_Popovici_2019,Hagerty_2018_np}. We note that, the oscillating weight in the energies studied in those works are different from the one studied here. Of particular interest for us in this article are \cite{Ansini_Braides_Chiado_2003,Hagerty_2018_np} where the case when the oscillations are much faster than the phase transition parameter \(\e\) is considered. We note that, in these references, the hypothesis that \(\delta=o_\e(1)\) is not sufficient to obtain a homogenization result and the authors assume in both papers that \(\delta=o_\e(\e^{3/2})\) even though they do not use the same techniques.

We will describe the asymptotic behaviour of minimizers of \eqref{eq:GL_pinned}, but also of a similar pinned Ginzburg-Landau functional in three dimensions and a pinned Allen-Cahn functional in $d$ dimensions, for arbitrary $d$.  

Let \(Q=(0,1)^d\) be the unit cube in \(\R^d\). We consider a function \(a_0 \in L^\infty(Q,\R)\) which satisfies
\begin{equation}\label{hyp:bounds_on_a_0}
\text{ there exist } 0<m<M \text{ such that } m< a_0(x) < M \text{ a.e.\ in } Q.
\end{equation}
Without loss of generality we can assume that
\begin{equation}\nonumber
\int_Q a_0(y) dy =1 \quad \text{ and } \quad m< 1 < M,
\end{equation}
and we will indicate how to adapt the arguments to the case $\M=\sqrt{\int_Q a_0}\neq1$.

\noindent We can see \(a_0\) as a \(1\)-periodic function  (still denoted by \(a_0\)) in \(\R^d\) by setting
\begin{equation}\nonumber
a_0(x)=a_0(x_1-\lfloor x_1 \rfloor, \cdots, x_d -\lfloor x_d \rfloor) \text{ for } x=(x_1,\cdots, x_d) \in \R^d,
\end{equation}
where \( \lfloor \cdot \rfloor \) denotes the integer part of a real number.
We first consider  the case where the pinning term is defined by 
\begin{equation}\label{def:a_e_periodic}
a_\e(x):=a_0\left( \frac{x}{\delta}\right) \quad \text{ with } \delta=\delta_\e\ll\e.
\end{equation}

We will also consider the case where the pinning term oscillates randomly. Let \( (\Omega, \Sigma,\mu)\) be a probability space. We assume that \(\R^d\) acts on \(\Omega\) by measurable isomorphisms and we denote by \(T\) this action. More precisely this means that for every \(x\in \R^d\), we have an application \( T(x): \Omega \rightarrow \Omega\) such that \(\mu[T(x)(A)]=\mu(A)\) for every set \(A\) in the \(\sigma\)-algebra \(\Sigma\), and we have that \(T(x+y)=T(x) \circ T(y)\) for every \(x,y\) in \(\R^d\).

\noindent We recall that a function \( a:\O \times \R^d\to\R\) is said to be \textit{stationary} with respect to the action \(T\) if \( a(\omega,x+y)=a(T(y)\omega,x)\) for every \(x,y\in \R^d\) and for almost every \(\omega \in \Omega\). A typical example of a stationary process is given by
\begin{equation}\label{eq:stationary_process}
\tilde{a}_0(\omega,x)=a_1(T(x)\omega) \quad \text{ with } a_1: \Omega \rightarrow \R \text{ a mesurable function.}
\end{equation}
We also recall that a function \(f: \Omega \rightarrow \R\) is \(T\)-invariant if \( f(T(x)\omega)=f(\omega)\) for every \(x\in \R^d\) and a.e.\ \(\omega \in \Omega\). The action \(T\) is \textit{ergodic} if every function that is invariant with respect to  \(T\) on \(\Omega\) is constant almost everywhere on \(\Omega\).

\noindent When considering a stationary-ergodic pinning term, we will assume that \({\tilde a_0}\) is given by \eqref{eq:stationary_process} with \(a_1\in L^\infty(\O,\R)\) which satisfies
\begin{equation}\label{def:a_1}
 m<a_1<M \text{ for some } 0< m< M.
\end{equation}
Without loss of generality we will assume that 
\begin{equation}\nonumber
\mathbb{E}(a_1)=1 \quad \text{ and } \quad 0<m< 1 < M,
\end{equation}
and we will indicate briefly how to adapt the argument to the case $\mathbb{E}(a_1)\neq1$.

\noindent Then, the pinning term will take the form
\begin{equation}\label{def:a_e_random}
a_\e(\omega,x)=\tilde{a}_0\left(\omega,\frac{x}{\delta}\right)=a_1\left(T\left(\frac{x}{\delta} \right)\omega \right) \quad \text{ with } \delta=\delta_\e\ll\e.
\end{equation}

Given a smooth bounded domain \(G\)  in \(\R^2\), we define the  (unpinned) Ginzburg-Landau energy of $(u,A)\in\H:=H^1(G,\mathbb{C})\times H^1(G,\R^2)$ by
\begin{equation}\label{eq:GL_classic}
GL_\e(u,A)=\frac12 \int_G |\nabla u-iAu|^2+\frac{1}{4\e^2}\int_G(1-|u|^2)^2 +\frac12 \int_G |\curl A-h_{\text{ex}}|^2.
\end{equation}

\noindent We also define the pinned energy without magnetic field
\begin{equation}\label{def:E_eps_pin}
E^{\text{pin}}_\e(u)= \frac12 \int_G |\nabla u|^2+ \frac{1}{4\e^2}\int_G (a_\e(x)-|u|^2)^2.
\end{equation}
Note that for any $(u,A)\in \H$ we have $GL_\e^{\text{pin}}(u,A)\geq E^{\text{pin}}_\e(|u|).$  We define the  denoised energy as
\begin{equation}\label{WGL}\widetilde{GL}_\e^{\text{pin}}(u,A):=GL_\e^{\text{pin}}(u,A)-\min_{U \in H^1(G)} E^{\text{pin}}_\e(U).\end{equation}
We will see in Corollary \ref{CorUniqSpecialSol} and Definition \ref{Def.SpecialSol} below that \( E^{\text{pin}}_\e(U)\) has a unique positive minimizer $U_\e$ in \(H^1(G)\).

Our main result is the following
\begin{theorem}\label{th:main1}
Assume that \( \delta=o_\e(\e)\) and that \(a_\e\) is given by \eqref{def:a_e_periodic} (respectively  \eqref{def:a_e_random}). Then \(U_\e\), the unique positive minimizer of $E^{\text{pin}}_\e$ in \(H^1(G)\), satisfies (respectively satisfies almost surely)
$$ \lim_{\e\to 0} \|U_\e-1\|_{L^\infty(G)} = 0.$$
Also, given   \((u_\e,A_\e)\), we have 
\begin{equation}\label{eq:expansion_main1}
\widetilde{GL}_\e^{\text{pin}}(u_\e,A_\e)=GL_\e(v_\e,A_\e)\left(1+O_\e\left(\|U_\e-1\|_{L^\infty(G)} \right)\right),\end{equation}
where $v_\v:=u_\v/U_\v$.

In particular \((u_\e,A_\e)\) is a family of quasi-minimizers of \( GL_\e^{\text{pin}}\) and \(\widetilde{GL}_\e^{\text{pin}}\) in \(\mathcal{H}\) {if and only if } \( (v_\e,A_\e)\) is a family of quasi-minimizers of the unpinned energy \(GL_\e\) in \(\mathcal{H}\). This equivalence holding only almost-surely in the stationary ergodic case.
\end{theorem}
By a family $(x_\e)$ of quasi-minimizers for some family of functionals $(F_\e)$ we mean a family which satisfies \(F_\e(x_\e)=(1+o_\e(1)) \inf F_\e\) as $\e\to 0$. 

Theorem \ref{th:main1} allows us to describe the behaviour of $u_\e$ as $\e\to 0$, and  in particular the behaviour of the vortices of \(u_\e\),  by using   the literature concerning the minimizers of \(GL_\e\).

\begin{remark}\label{RmarkUnifConvMNeq1} If we do not assume that \( \mathcal{M}:=\sqrt{\int_Q a_0(y) dy} =1\) in the periodic case, or that \( \mathcal{M}:=\sqrt{\mathbb{E}(a)}=1\) in the random case, then Theorem \ref{th:main1} has to be modified as follows. With the same assumptions and notations we have that \( (v_\e,A_\e)\) is a family of quasi-minimizers of $GL^{\mathcal{M}}_\e$ (instead of $GL_\e$) where for $(v,A)\in\H$
\begin{eqnarray}\nonumber
GL^{\mathcal{M}}_\e(v,A)&=&GL^{\mathcal{M}}_{\e,h_{\rm ex},G}(v,A)
\\\nonumber&:=&\frac{\M^2}{2} \int_G |\nabla v-iAv|^2+\frac{\M^4}{4\e^2}\int_G(1-|v|^2)^2 +\frac12 \int_G |\curl A-h_{\text{ex}}|^2
\\\nonumber&=&\frac{\M^2}{2} \int_{\M\cdot G} |\nabla v'-iA'v'|^2+\frac{1}{2\e^2}(1-|v'|^2)^2 + |\curl A'-h_{\rm ex}/\M^2|^2
\\\nonumber&=&\M^2 GL_{\e,h_{\rm ex}/\M^2,\M\cdot G}(v,A)
\end{eqnarray}
where $A'(\cdot)=A(\cdot/\M)/\M$, $v'(\cdot)=v(\cdot/\M)$ and the expansion \eqref{eq:expansion_main1} holds. The study of \(GL^{\mathcal{M}}_{\e,h_{\rm ex},G}\) reduces to the study of \(GL_{\v,h_{\rm ex}/\M^2,\M\cdot G}\) in the dilated domain \(\mathcal{M} \cdot G\) by a change of variable and a dilation of the unknowns. This is similar to the operations made in the nondimensionalizing process of the GL functional (see e.g.\ section 2.1.1 in \cite{Sandier_Serfaty_2007}).
\end{remark}

This article is organized as follows. In section \ref{sec:substituion_lemma} we recall the decomposition Lemma from  \cite{Lassoued_Mironescu_1999} and we show how it reduces the study of the problem to the convergence of a minimizer of \(E^{\text{pin}}_\e\) in \(H^1(G)\), with \(E^{\text{pin}}_\e\) being defined in \eqref{def:E_eps_pin}. In particular we show that there exists a unique positive minimizer $U_\v$ of \(E_\e^{\text{pin}}\) in \(H^1(G)\). In section \ref{sec:convergence} we prove the convergence, in \(L^\infty\) norm, of this minimizer to the square root of the average of \(a_0\) in the periodic case and to the square root of the expectation of \(a_1\) in the random case. The proofs of both results make use of a blow-up argument. This convergence is sufficient to prove Theorem \ref{th:main1}. We also give another proof of the convergence of \(U_\v\) which has the advantage of working in Lipschitz bounded domains and giving explicit rates of convergence but the disadvantage of requiring \(\delta=O(\e^2)\). Then we use Theorem \ref{th:main1} and known results in the literature to describe the behavior of minimizers of \(GL_\e^{\text{pin}}\) in the two-dimensional case in section \ref{sec:2D_results}. We use also analogous results to Theorem \ref{th:main1} in three dimensions in section \ref{sec:3D_results} and also for the Allen-Cahn problem with prescribed mass in section \ref{sec:Allen-Cahn_results}.

\section{The decomposition lemma}\label{sec:substituion_lemma}
In the framework of pinned Ginzburg-Landau type energies a useful decomposition method is described in \cite{Lassoued_Mironescu_1999}. This can be expressed by the following lemma.
\begin{lemma}\label{lem:substitution}[Decomposition Lemma]
Let \(G\) be a Lipschitz domain of \(\R^d\) with \(d\geq 1\). Let \(p \in L^\infty(G,\R^+)\) and let us assume that \(U\) is a solution of
\begin{equation}\label{eq:GL_Neumann_p}
\left\{
\begin{array}{rcll}
-\Delta U &=& \dfrac{1}{\e^2}U(p-|U|^2) & \text{ in } G \\
\p_\nu U &=&0  & \text{ on } \p G,
\end{array}
\right.
\end{equation}
which satisfies \(U\geq m>0\) in \(G\) for some $m>0$. 
 We consider \begin{equation}\nonumber
E_\e^p(u)= \frac12 \int_G |\nabla u|^2+\frac{1}{4\e^2}\int_G(p-|u|^2)^2.
\end{equation}
and for \( d=2 \)  we define
\begin{equation}\nonumber
GL_{\e}^p(u,A)= \frac12 \int_G |\nabla u-iAu|^2+\frac{1}{4\e^2}\int_G (p(x)-|u|^2)^2 +\frac12 \int_G |\curl A-h_{ex}|^2.
\end{equation}

 Then, for every \(d \geq 1\) if we set \( u= U v\) we obtain
 \begin{equation}\label{eq:lem_subst_2}
E_\e^p(u)=E_\e^p(U)+\frac12 \int_G U^2|\nabla v|^2+\frac{1}{4\e^2}\int_G U^4(1-|v|^2)^2.
\end{equation}
For \(d=2\), with \( u= U v\) we find
\begin{multline}\label{eq:lem_substitution}
GL_{\e}^p(u,A) =E_\e^{p}(U) +\frac12 \int_G U^2|\nabla v-iAv|^2+\frac{1}{4\e^2}\int_G U^4(1-|v|^2)^2 \\ 
+\frac12 \int_G |\curl A-h_{ex}|^2.
\end{multline}
\end{lemma}

\begin{remark}
In \cite{Lassoued_Mironescu_1999}, the decomposition Lemma was proved in the context of a planar pinned Ginzburg-Landau type energy without magnetic field (given by \eqref{def:E_eps_pin} where $a_\v$ is independent of $\v$). In this context the Ginzburg-Landau type energy studied is minimized under a Dirichlet boundary condition $g\in \C^\infty(\p G,\S^1)$ and with \(G\) a smooth bounded domain. The authors of \cite{Lassoued_Mironescu_1999} proved the decoupling \eqref{eq:lem_subst_2} with  $U_{Dir}\in H^1(G,\R^+)$ instead of $U$ where $U_{Dir}$ is the unique minimizer of $E_\e^p$ submitted to the boundary condition $U_{Dir}=1$ on $\p G$ (they also need that  $p=1$ on $\p G$ in the sense of traces and $p\geq m>0$ in $G$). If we look at their proof we can see that the minimality of $U_{Dir}$ is only used through the validity of the Euler-Lagrange equation $-\Delta U_{Dir}=\v^{-2}U_{Dir}(p-U_{Dir} ^2)$, while the boundary condition $U_{Dir}=1$ on $\p G$ makes the boundary terms vanish since $U_{Dir}-p=0$ on $\p G$. These boundary terms may also be cancelled with an homogeneous Neumann boundary condition as in \eqref{eq:GL_Neumann_p}. Hence, one may follow (in arbitrary dimension) the argument of \cite{Lassoued_Mironescu_1999} to prove \eqref{eq:lem_subst_2} and also \eqref{eq:lem_substitution}.
\end{remark}

\begin{corollary}\label{CorUniqSpecialSol}Let \(G\) be a Lipschitz domain of \(\R^d\) with \(d\geq 1\), $p\in L^\infty(G,\R)$ and $\v>0$. Assume that $p\geq m >0$. Then there exists a unique minimizer $U_\e$ of  $E_\e^p$  in $H^1(G,\mathbb{C})$ which is nonnegative. It satisfies
$$\|p\|_{L^\infty(G)}\geq|U|\geq m.$$ Any other minimizer is of the form $\alpha U_\e$ for some $\alpha\in\S^1$. Moreover $U_\e$ is the only solution of \eqref{eq:GL_Neumann_p} which is positive in \(\overline{G}\).
\end{corollary}

\begin{proof}
Let $\v>0,p\in L^\infty(G,\R)$ be such that  $p\geq m$ for some $m>0$. Let $U$ be a minimizer of  $E_\e^p$  in $H^1(G,\mathbb{C})$. It is clear that  $E^p_\v(|U|)\leq E_\v^p(U)$. Moreover $V=\max(|U|,m)$ satisfies $E^p_\v(V)\leq E_\v^p(|U|)$. Thus,  by minimality of $|U|$,  $E^p_\v(V)= E_\v^p(|U|)$. This implies
\[ \int_{\{|U|< m\}} |\nabla |U||^2=\frac{1}{4\e^2}\int_{ \{|U|< m\}}\left( (p-m^2)^2-(p-|U|^2)^2 \right)\leq 0\]
and then \(\{|U|<m\}=\emptyset\).
 By considering $V'=\min(|U|,\|p\|_{L^\infty(G)})$ we find $|U|\leq \|p\|_{L^\infty(G)}$.

Assume $U'$ is another minimizer of $E_\e^p$  in $H^1(G,\mathbb{C})$. From Lemma \ref{lem:substitution}, letting $v=U'/|U|$ we have
\[
0=E^p_\v(U')-E^p_\v(|U|)=\dfrac{1}{2}\int_G |U|^2|\n v|^2+\dfrac{|U|^4}{2\v^2}(1-|v|^2)^2.
\]
Hence $v$ is a constant and $v\in\S^1$.

To prove the last statement, we consider $V$ solution of \eqref{eq:GL_Neumann_p} which is positive in \(\overline{G}\). There exists $m'>0$ such that  $V\geq m'$. Using Lemma \ref{lem:substitution} again, letting $v=|U|/V$ we have
\[
0\geq E^p_\v(|U|)-E^p_\v(V)=\dfrac{1}{2}\int_G V^2|\n v|^2+\dfrac{V^4}{2\v^2}(1-|v|^2)^2.
\]
Since $v\geq0$ we deduce $v=1$, {\it i.e.} $V=|U|$.
\end{proof}
From Corollary \ref{CorUniqSpecialSol}, one may state the following definition.
\begin{definition}\label{Def.SpecialSol}
Let \(G \subset \R^d\) be a Lipschitz bounded domain. For $p\in L^\infty(G,\R)$ such that $p\geq m$ for some $m>0$ and for $\v>0$ we let $U_\v^p$ be  the unique positive minimizer of $E_\v^p$ in $H^1(G,\mathbb{C})$. It satisfies \(m \leq U^p_\e\leq \|p\|_{L^\infty}\) in \(G\). Moreover, if \(G\) is a \(\mathcal{C}^2\) bounded domain, then  by elliptic regularity,  \(U^p_\e \in W^{2,q}(G)\) for every \(1\leq q<+\infty\). If $p=a_\e$ is given by \eqref{def:a_e_periodic} or \eqref{def:a_e_random} then we simply write $U_\v$.
\end{definition}

Next we give a Lipschitz estimate on \( U_\e^p\) in the case where \(G\) is a \(\C^1\) domain. The proof follows the argument of Lemma A.2 in \cite{Bethuel_Brezis_Helein_1993}.

\begin{lemma}\label{Lem.LipBound} Let us assume that $G$ is a \(\C^1\) bounded domain. Let $\v>0$ and let $p\in L^\infty(G,\R)$ such that $p\geq m$ for some $m>0$. There is $C>1$ depending only on $G$ and $\|p\|_{L^\infty}$ such that
\begin{equation}\label{eq:gradient_estimate}
\|\nabla U^p_\e\|_{L^\infty(G)} \leq \frac{C}{\e}.
\end{equation}
\end{lemma}

\begin{proof}
Following step by step the proof of Lemma A.2 in \cite{Bethuel_Brezis_Helein_1993} we obtain that if \( u \in H^1(G)\) and \(f \in L^\infty(G)\) with \(\int_G f(x) {\rm d}x =0\) satisfy 
\begin{equation}\nonumber
\left\{
\begin{array}{rcll}
-\Delta u &=& f & \text{ in } G \\
\p_\nu u &=& 0 & \text{ on } \p G
\end{array}
\right.
\end{equation}
then there exists \(C>0\) which depends only on \(G\) such that \[ \|\nabla u\|_{L^\infty(G)} \leq C \|f\|_{L^\infty(G)}\|u\|_{L^\infty(G)}.\]
The only modification with respect to the proof of lemma A.2 in \cite{Bethuel_Brezis_Helein_1993} is the use of the following Neumann elliptic estimate in place of its Dirichlet counterpart: let \(A \in \C^1(B_1^+,\mathcal{M}_d(\R))\) which is bounded and uniformly elliptic where \(B_1^+=\{x \in B_1(0) \mid  x_d>0\}\), let \(g \in L^\infty( B_1^+)\) and \(v\in H^1(B_1^+)\) satisfying
\begin{equation}\nonumber
\left\{
\begin{array}{rcll}
-\dive(A(x)\nabla v)&=&g & \text{ in } B_1^+ \\
\p_\nu v &=&0 & \text{ on } B_1\cap \{ x \in B_1^+ \mid x_d=0\}
\end{array}
\right.
\end{equation}
then 
\(\|\nabla v\|_{L^\infty(B_{1/2}^+)} \leq C (\|g\|_{L^\infty(B_1^+)}+\|v\|_{L^\infty(B_1^+)})\) for some \(C>0\) depending on the ellipticity constant of \(A\) and on \(\|A\|_{\C^1(B_1^+)}\).
\end{proof}
\section{Convergence of the free minimizer}\label{sec:convergence}

\subsection{The periodic case}
Let \(E^{\text{pin}}_\e\) be defined by \eqref{def:E_eps_pin} with \(a_\e\) being defined by \eqref{def:a_e_periodic}.

\begin{theorem}\label{th:main2} Let \(G \subset \R^d\) be a  \(\C^1\) bounded domain. Let \(U_\e\) be the minimizer of \(E_\e\) in \(H^1(G)\) given by Definition \ref{Def.SpecialSol}. Then
\begin{equation}\label{eq:conv_to_0}
\lim_{\e \to 0} \|U_\e-\M\|_{L^\infty(G)}=0.
\end{equation}
\end{theorem}
\noindent We recall that, for simplicity, we assumed that \(\mathcal{M}=\sqrt{\int_Q a_0(x){\rm d }x}=1\).

\begin{proof}
By contradiction we assume that \eqref{eq:conv_to_0} is not true. Then there exists \(\eta >0\) and a sequence of points \((x_{\e})_{\e>0}\) such that \( |U_\e(x_\e)-1|\geq \eta\) for all \( \e>0\) small enough. \\

We first assume that \( \rho_\e:=\dist(x_\e,\p G)\gg \e \). We then consider the blow-up function \( V_\e(y)= U_\e (x_\e+\e y)\) defined for $y\in B(0,\rho_\e/\v)$. This function satisfies
\begin{equation}\label{eq:blow-up_eps}
-\Delta V_\e = V_\e (b_\e-V_\e^2)  \text{ in } B(0,\rho_\e/\v)
\end{equation}
with \(b_\e(y):=a_\e(x_\e+\e y)=a_0\left(\frac{x_\e+\e y}{\delta} \right)\) for $y\in B(0,\rho_\e/\v)$. 

\noindent Besides  \(m\leq V_\e\leq \|U_\v\|_{L^\infty(G)}\leq M\) and, by the Lipschitz estimate \eqref{eq:gradient_estimate}, we have that \(V_\e\) satisfies \( \|\nabla V_\e \|_{L^\infty (B(0,\rho_\e/\v))} \leq C\).  

\noindent By the Arzela-Ascoli theorem, up to passing to a subsequence, there is $V_0:\R^d\to[m,M]$ such that \(V_\e \rightarrow V_0\) locally uniformly in \(\R^d\). Since \(\e/\delta \rightarrow +\infty\), the strong oscillations of \(b_\e\) implies that 
\begin{equation}\label{ConvCarreOscill}
b_\e \rightharpoonup  \int_Q a_0(x){\rm d}x=1 \text{ in } L_{\text{loc}}^1(\R^d),
\end{equation}
(see e.g.\ chapter 2 in \cite{Cioranescu_Donato_1999} or p.5 in \cite{Jikov_Kozlov_Oleunik_1994}). Thus we find that \( V_\e b_\e \rightharpoonup V_0 \) in \(\mathcal{D}'(\R^d)\) and \(V_\e^3 \rightarrow V_0^3\) locally uniformly  in \(\R^d\). Passing to the limit in the sense of distributions in \eqref{eq:blow-up_eps} we find that the limit \(V_0\) satisfies
\begin{equation}\label{EqLimEspa}
-\Delta V_0=V_0(1-V_0^2) \text{ in } \R^d.
\end{equation}
Since we have that \(m\leq V_0 \leq M\), by using Theorem 2.1 in \cite{Farina_2003} we conclude that \(V_0 \equiv 1\). Thus \(V_\e(0)=U_\e(x_\e) \rightarrow 1\) which is a contradiction. \\

Now we assume that, up to passing to a subsequence, \( \dist(x_\e,\p G)=O \left( \e\right)\). Thus we may define $y_\v:=\Pi_{\p G}(x_\v)$, the orthogonal projection of $x_\v$ on $\p G$. We then have $|x_\v-y_\v|=O(\v)$. Up to passing to a further subsequence we may assume that $y_\v\to y_0\in\p G$. We set  \( V_\e(y):= U_\e (y_\e+\e y)\) for $y\in B_\v^+=\dfrac{B(y_0-y_\v,1)\cap(G-y_\v)}{\v}$. By using that $G$ is \(\C^1\), up to passing to a subsequence and up to considering a vectorial rotation, we may assume that for all \(x\in \R^d_+\) there exists \(\e_0>0\) such that \(x \in B_\e^+\) for all \(\e <\e_0\).  Then, as in the first case, we can obtain the existence of $V_0:\R^d_+\to\R$ such that, up to a subsequence, $V_\v\to V_0$ locally uniformly in \(\R^d_+\).  Passing to the limit in \eqref{eq:blow-up_eps} we find that \(V_0\) satisfies
\begin{equation}\label{LimitEqProxBound}
\left\{
\begin{array}{rcll}
-\Delta V_0&=& V_0(1-V_0^2) &\text{ in } \R^d_+ \\
\p_\nu V_0 &=& 0 & \text{ on } \p \R^d_+.
\end{array}
\right.
\end{equation}
We can consider a new function defined for $y=(y',y_d)\in\R^{d-1}\times\R$ by
\begin{equation*}
\tilde{V}_0(y)=\begin{cases} V_0(y) &\text{ if } y_d\geq 0 \\
V_0(y',-y_d) & \text{ if } y_d<0.
\end{cases}
\end{equation*} 
We can check that \(\tilde{V}_0\) satisfies \(-\Delta \tilde{V}_0=\tilde{V}_0(1-\tilde{V}_0^2)\) in \(\R^d\) and we conclude as before that \(\tilde{V}_0 \equiv 1\). On the other hand, since $|x_\v-y_\v|=O(\v)$, up to passing to a subsequence, there exists $y_\star\) in \(\R^d_+\) such that \( y_\star:=\lim_{\v \to 0}\dfrac{x_\v-y_\v}{\v}$  and $|\tilde{V}_0(y_\star)-1|=\lim_{\v \to 0}|{V}_\v(y_\v)-1|\geq\eta>0$. This is a contradiction and this concludes the proof of the theorem.
\end{proof}
\begin{remark}
 We recall that, for simplicity, we assumed that \(\mathcal{M}=\sqrt{\frac{1}{|Q|}\int_Q a_0(x){\rm d}x}=1\). If $\M\neq1$ then we change the definition of $V_\v\& b_\v$ by letting 
 \[
 V_\v(y)=\begin{cases}U_\v(x_\v+\v y/\M^2)/\M&\text{if }\dist(x_\e,\p \O)\gg\v\\U_\e(y_\v+\v y/\M^2)/\M&\text{if }\dist(x_\e,\p \O)=\mathcal{O}(\v)\end{cases}
 \]
 and 
 \[ b_\v(y)=\begin{cases}a_\v(x_\v+\v y/\M^2)/\M^2&\text{if }\dist(x_\e,\p \O)\gg\v\\a_\e(y_\v+\v y/\M^2)/\M^2&\text{if }\dist(x_\e,\p \O)=\mathcal{O}(\v).\end{cases}
 \]
Convergence \eqref{ConvCarreOscill} reads $b_\e \rightharpoonup 1\text{ in } L_{\text{loc}}^1(\R^d)$ and thus  \eqref{EqLimEspa}$\&$\eqref{LimitEqProxBound} still hold. The rest of the proof is unchanged.
\end{remark}
We can now give a proof of Theorem \ref{th:main1} in the periodic case.

\begin{proof}[Proof of Theorem \ref{th:main1} in the periodic case]
Let $(u,A)\in \H$ and $\v>0$. We recall that \(\widetilde{GL}_\e^{\text{pin}}\) is defined in \eqref{WGL}. Letting $v=u/U_\v$, with \eqref{eq:lem_subst_2} we have
\begin{eqnarray*}
\widetilde{GL}_\e^{\text{pin}}(u,A)&=&\frac12 \int_G U_\v^2|\nabla v-iAv|^2+\frac{1}{4\e^2}\int_G U_\v^4(1-|v|^2)^2+\frac12 \int_G |\curl A-h_{ex}|^2
\\&=:&{GL}_\e^{\rm{weight}}(v,A).
\end{eqnarray*}
By using that \( 0<m\leq 1\leq M\) we obtain:
\[
m^4\times{GL}_\e(v,A)\leq{GL}_\e^{\rm{weight}}(v,A)\leq M^4\times{GL}_\e(v,A).
\]
Therefore 
\[
m^4\times\inf_\H{GL}_\e\leq \inf_\H{GL}_\e^{\rm{weight}}\leq M^4\times\inf_\H{GL}_\e.
\]
Moreover, if $(v,A)\in\H$ is such that ${GL}_\e(v,A)\leq 2M^4\times\inf_\H{GL}_\e$ then 
\begin{equation}\label{ComparisiondifeEn}
|{GL}^{\rm{weight}}_\e(v,A)-{GL}_\e(v,A)|\leq2M^4\times\max(\|U_\v^2-1\|_{L^\infty},\|U_\v^4-1\|_{L^\infty})\inf_\H{GL}_\e.
\end{equation}
In particular, with Theorem \ref{th:main2}, we get when $\v\to0$:
\begin{equation}\label{ComparisiondifeEnBis}
\inf_\H{GL}_\e=(1+o_\v(1))\inf_\H{GL}_\e^{\rm{weight}}.\end{equation}
Now let $\v\to0$, let $(u_\v,A_\v)\in\H$ be a family of configuration and write $v_\v=u_\v/U_\v$. We have
\begin{eqnarray*}
&&\text{\((u_\e,A_\e)\) is a family of quasi-minimizers of \(\widetilde{GL}_\e^{\text{pin}}\)}
\\&\Leftrightarrow&\widetilde{GL}_\e^{\text{pin}}(u_\e,A_\e)=(1+o_\v(1))\times\inf_\H\widetilde{GL}_\e^{\text{pin}}
\\&\Leftrightarrow&{GL}_\e^{\rm{weight}}(v_\e,A_\e)=(1+o_\v(1))\times\inf_\H{GL}_\e^{\rm{weight}}
\\\text{(with \eqref{ComparisiondifeEn}$\&$\eqref{ComparisiondifeEnBis})}&\Leftrightarrow&{GL}_\e(v_\e,A_\e)=(1+o_\v(1))\times\inf_\H{GL}_\e
\\&\Leftrightarrow&\text{\((v_\e,A_\e)\) is a family of quasi-minimizers of \({GL}_\e\)}.
\end{eqnarray*}
\end{proof}
\subsection{Rate of convergence in special cases}
Although Theorem \ref{th:main2} has the advantage of working in every dimension and its proof can be extended to random stationary ergodic pinning term, it has the disadvantage of requiring \(\C^1\) regularity for \(\p G\) and of not giving a rate of convergence, which could be useful in some regime of applied magnetic fields \(h_{ex}\). 

In the following, we give rates of convergence under additional assumptions. We work in dimension \(2\) and the assumptions take two forms: either an assumption on $\delta = \delta(\v)$ or a symmetry assumption on $a_0$.

\begin{proposition}\label{prop:speed_convergence}
Let \(G \subset \R^2\) be  a Lipschitz  bounded domain. Let \(U_\e\) be the minimizer of \(E_\e^\pin\) given by Definition \ref{Def.SpecialSol}. There exists \(C>0\) (independent of \(\delta\) and \(\e\)) such that
\begin{equation}\label{eq:ccl_weaker}
\|U_\e-1\|_{L^2(G)} \leq C\left( \frac{\delta}{\e} +\sqrt{\delta} \right).
\end{equation}
If we assume furthermore that
\begin{equation}\label{eq:assumption_stronger}
\delta =O_\e( \e^2).
\end{equation}  
 then
\begin{equation}\label{eq:ccl_stronger}
 \|U_\e -1\|_{L^\infty(G)}=O_\e\left[ \left(\frac{\delta}{\e^2} \right)^{1/4} \right].
\end{equation}
\end{proposition}

Note that if we assume that \(\delta =O_\e(\e^2)\) then \eqref{eq:ccl_weaker} becomes \(\|U_\e-1\|_{L^2(G)} =O_\e(\sqrt{\delta})\).
 Before proving Proposition \ref{prop:speed_convergence} we present an estimate which may be obtained with a weaker assumption. 
\begin{remark}
We may get an explicit speed of convergence with assumptions weaker than $\delta =O_\e(\e^2)$. Namely, let $\chi:=\delta/\v$ and consider the assumption
\begin{equation}\label{SharpEst}
\chi^{1/8}\times{\rm e}^{\chi^{1/8}|\ln\delta|}\to0.
\end{equation}
It is clear that \(\delta =O_\e(\e^2)\) (which reads $\chi=O_\e(\v)$) implies \eqref{SharpEst}. One may prove that if \eqref{SharpEst} holds we have 
\begin{equation}\label{TxCvRk}
\|U_\v-1\|_{L^\infty(G)}<4\chi^{1/8}=4 \left(\frac{\delta}{\e^2} \right)^{1/8}.
\end{equation}
The proof of \eqref{TxCvRk} is quite long and thus it is omitted here.
\end{remark}

\begin{proof}[Proof of Proposition \ref{prop:speed_convergence}]
We let \(\chi:=\frac{\delta}{\e}\). For \(0<\e<1\) we consider the energy
\begin{equation}
\hat{E}_\e(u)=\frac{1}{2}\int_Q |\nabla u|^2+\frac{\chi^2}{4}\int_Q(a_0-|u|^2)^2
\end{equation}
defined for \(u\in H^1(Q, \mathbb{C})\). Recall that $Q=(0,1)^2$ is the unit square.

From Corollary \ref{CorUniqSpecialSol}  there exists a unique positive minimizer \(\hat{U}_\e\) of \(\hat{E}_\e\) in \(H^1(Q)\) and \( m\leq \hat{U}_\e \leq M\). This minimizer satisfies

\begin{equation}\label{eq:equation_sur_Q}
\left\{
\begin{array}{rcll}
-\Delta \hat{U}_\e &=& \chi^2 \hat{U}_\e (a_0-\hat{U}_\e^2) &\text{ in } Q \\
\p_\nu \hat{U}_\e &=& 0 & \text{ on } \p Q.
\end{array}
\right.
\end{equation}
We set \( \ell_\v=\ell := \int_Q \hat{U}_\e\). By using the homogeneous Neumann boundary condition to extend \(\hat{U}_\e\) in the square \( (-1,2)^2\), we can use interior elliptic estimates in \(Q\) and obtain that, for all \(2 \leq p< +\infty\)
\begin{equation}\label{eq:estimation_sur_Q}
\|\hat{U}_\e-\ell\|_{W^{1,p}(Q)}=O_\e(\chi^2).
\end{equation}
By multiplying \eqref{eq:equation_sur_Q} by \(\hat{U}_\e\), integrating by parts and using \eqref{eq:estimation_sur_Q} we find
\begin{equation}\label{eq:equipartition}
\chi^2 \int_Q \hat{U}_\e^2(a_0-\hat{U}_\e^2)=\int_Q |\nabla \hat{U}_\e|^2= O_\e(\chi^4).
\end{equation}
We infer that
\begin{equation}\nonumber
\int_Q \left( \ell^2+O_\e(\chi^2)\right)\left( a_0-\ell^2+O(\chi^2) \right) =O_\e(\chi^2).
\end{equation}
Since \(\ell\geq m>0\) we obtain that \(\ell^2 =\int_Q a_0(x){\rm d}x+O_\e(\chi^2)=1+O_\e(\chi^2)\). Thus we can reformulate \eqref{eq:estimation_sur_Q} in: for all \( 2 \leq p< +\infty\),
\begin{equation}\label{eq:estimation_Q_2}
\|\hat{U}_\e-1\|_{W^{1,p}(Q)} =O_\e(\chi^2).
\end{equation}
We thus obtain from \eqref{eq:equipartition} that
\begin{equation}\label{eq:estimation_energie_cell_Y}
\hat{E}_\e(\hat{U}_\e,Q)=\hat{E}_\e(1,Q)+O_\e(\chi^4).
\end{equation}

Now let \( U_\e\) be the minimizer of \(E_\e\) in \(H^1(G)\) given by Definition \ref{Def.SpecialSol}. For \(k,l\in \mathbb{Z}\) such that \(Q_{k,l}:=\delta(k,l)+\delta Q \subset G\) we define \( \tilde{U}_\e^{k,l}: Q \rightarrow \R\) by \(\tilde{U}_\e^{k,l}(y)=U_\e \left(\delta (y+(k,l))\right)\). By minimality of \(\hat{U}_\e\) in \(H^1(Q)\) we have
\begin{equation}\label{eq:estimation_cell_Ydelta}
\hat{E}_\e(\tilde{U}_\e^{k,l},Q) \geq \hat{E}_\e(\hat{U}_\e,Q)=\hat{E}_\e(1,Q)+O_\e(\chi^4).
\end{equation}

We can then decompose \(G\) in cells \(Q_{k,l}\). We denote by \(N_\delta\) the number of cells \(Q_{k,l}\) included in \(G\). We have
\[N_\delta =\frac{|G|}{\delta^2}+O_\e\left(\frac{1}{\delta} \right). \]
We also denote \(G_\delta:=G \setminus \cup_{Q_{k,l} \subset G} Q_{k,l}\) and we can see that \(|G_\delta|=O_\e(\delta)\). We have
\begin{align*}
E_\e^\pin(U_\e,G) &\geq  \sum_{\substack{Q_{k,l} \subset G}} E_\e^\pin(U_\e, Q_{k,l})   \geq \sum_{\substack{Q_{k,l}\subset G}} \hat{E}_\e(\tilde{U}_\e^{k,l},Q) \\
& \geq \sum_{\substack{Q_{k,l}\subset G}} \left(\hat{E}_\e(1,Q)+O_\e\left(\frac{\delta^4}{\e^4} \right) \right) \\
& = O_\e\left(\frac{\delta^2}{\e^4} \right)+\sum_{\substack{Q_{k,l}\subset G} } E_\e^\pin(1,Q_{k,l}). 
\end{align*}
But, we observe that 
\begin{equation}\nonumber
\sum_{Q_{k,l}\cap G_\delta \neq \emptyset} E^\pin_\e(1,Q_{k,l}) =O_\e\left( \frac{\delta}{\e^2} \right).
\end{equation}
Hence we obtain,
\begin{align*}
E_\e^\pin(U_\e,G) &\geq E_\e^{\text{pin}}(1,G)- \sum_{Q_{k,l}\cap G_\delta \neq \emptyset} E^\pin_\e(1,Q_{k,l})+O_\e\left( \frac{\delta^2}{\e^4}\right) \\
& \geq  E_\e^{\text{pin}}(1,G) +O_\e\left( \frac{\delta^2}{\e^4}+ \frac{\delta}{\e^2}\right).
\end{align*}

Thus, since \(E_\e^\pin(1,G)\geq E_\e^\pin(U_\e,G)\), we find that 
\begin{equation}\label{eq:diff_energy}
E_\e^\pin(1)-E_\e^\pin(U_\e)=O_\e\left(\frac{\delta^2}{\e^4} + \frac{\delta}{\e^2}\right).
\end{equation}
Now we use Lemma \ref{lem:substitution}, writing \(1=U_\e v\), to get
\begin{equation}\label{eq:substitution_M_U}
E_\e^\pin(1)=E_\e^\pin(U_\e)+\frac12 \int_G U_\e^2|\nabla v|^2+\frac{1}{4\e^2}\int_G U_\e^4(1-v^2)^2.
\end{equation}
We deduce from \eqref{eq:substitution_M_U},\eqref{eq:diff_energy} and the fact that \(U_\e\geq m>0\) that
\begin{equation}\label{estimate}
\frac12 \int_G|\nabla v|^2+\frac{1}{4\e^2}\int_G (1-v^2)^2 =O_\e\left(\frac{\delta^2}{\e^4} + \frac{\delta}{\e^2}\right).
\end{equation}
Hence we find that \(\int_G (1-v^2)^2=O_\e \left(\frac{\delta^2}{\e^2}+\delta \right)\) which implies \eqref{eq:ccl_weaker}.
 Now we assume that \(\delta=O_\e(\e^2)\). This implies \( \|U_\e-1\|_{L^2(G)}=O_\e(\sqrt{\delta})\). To prove the \(L^\infty\) estimate we argue by contradiction. We assume that there exist two sequences \(\e_n \rightarrow 0\) and \( (x_n)_n \subset G\) such that
\begin{equation}
1-v(x_n)^2 \geq (n+1) \left(\frac{\delta_n}{\e_n^2} \right)^{1/4}.
\end{equation}
We set \(\rho_n:=\delta_{n}/\e_n^2 \). Since \( \|\nabla v\|_{L^\infty(G)}\leq \frac{M}{m}\|\nabla U_\v\|_{L^\infty(G)}=O_\e\left(\frac{1}{\e} \right)\) we find that there exists \(c>0\) independent of \(\e\) such that
\begin{equation}
1-v^2(x) \geq n \rho_n^{1/4} \text{ for every } x\in B(x_n,c\e_n \rho_n^{1/4})\cap G.
\end{equation}
By Lipschitz regularity of \(G\) we can assume that \(c>0\) is small enough (independent of \(\e_n\)) so that \( |B(x_n,c\e_n \rho_n^{1/4})\cap G| \geq c^3 \e_n^2 \rho_n^{1/2} \).
We then have
\[ \int_{B(x_n,c\e_n \rho_n^{1/4})\cap G} (1-v^2)^2 \geq c^3 n^2 \e_n^2 \rho_n .\]
By using that, from \eqref{estimate} and \eqref{eq:assumption_stronger}, \(\int_G (1-v^2)^2=O_\e(\delta)\) we arrive at \(\delta_n \geq c n^2  \e_n^2 \left(\frac{\delta_{n}}{\e_n^2} \right)= cn^2 \delta_n\) for some \(c>0\) sufficiently small (independent of \(\e\)) and for all \(n \in \mathbb{N}\). This is a contradiction and then we find that \(\|1-v\|_{L^\infty(G)}=O_{\e_n}\left(\left(\frac{\delta_{n}}{\e_n^2} \right)^{1/4} \right)\) which implies the second part of \eqref{eq:ccl_stronger}.
\end{proof}

In some cases we can improve the \(L^\infty\) bound obtained in the previous proposition. For example, we make the following symmetry assumption on \(a_0:Q \rightarrow \R\)
\begin{equation}\label{SymetryOfa_0}
\begin{cases}a_0(\frac{1}{2}-x_1,x_2)=a_0(x_1,x_2),& \forall\, (x_1,x_2)\in (0,\frac{1}{2})\times (0,1), \\
 a_0(x_1,\frac12 -x_2)=a_0(x_1x_2),& \forall\, (x_1,x_2) \in (0,1)\times (0,\frac12).
 \end{cases}
\end{equation}

\begin{proposition}\label{PropSymEstim}
Assume that \(G\) is a square in \(\R^2\) of size \(L\). Let \(\delta_n:=\frac{L}{n} \rightarrow 0\) and \(\e_n \rightarrow 0\) be such that \(\delta_n=o_{\e_n}(\e_n)\). Let $a_{\v_n}$ be defined by \eqref{def:a_e_periodic} on a $\delta_n\times \delta_n$ grid matching with $G$ and assume that \eqref{SymetryOfa_0} holds.

Let \(U_{\e_n}\) be the positive minimizer of \(E_{\e_n}^\pin\) given by Definition \ref{Def.SpecialSol}. Then there exists \(C>0\) such that
\begin{equation}
\|U_{\e_n}-1\|_{L^\infty(G)} \leq C \frac{\delta_n^2}{\e_n^2}, \quad \|\nabla U_{\e_n}\|_{L^\infty(G)}\leq \frac{C}{\e_n}.
\end{equation}
\end{proposition}
\begin{remark}
Proposition \ref{PropSymEstim} is still valid for a polygonal domain $G$ such that $G$ matches with the union of cells of  $\delta_n\times\delta_n$ grids with $\delta_n\to0$.
\end{remark}
\begin{proof}

We drop the subscript \(n\) for simplicity. We decompose the domain \(G\) in small regular cells of size \(\delta\) which we denote by \(Q_{k,l}\) for \(k,l \in \mathbb{Z}\). Let \( \hat{U}_\e\) be the positive minimizer of \(\hat{E}_\e(u)=\frac12 \int_Q |\nabla u|^2+\frac{\delta^2}{4\e^2}\int_Q (a_0(x)-|u|^2)^2\) in \(H^1(Q)\). Note that \(\hat{U}_\e\) satisfies \eqref{eq:equation_sur_Q}.  We claim that
\begin{equation}
\tr_{|\{0\}\times(0,1)}\hat{U}_\e =\tr_{|\{1\}\times(0,1)} \hat{U_\e} \text{ and } \tr_{|(0,1)\times \{0\}}\hat{U}_\e =\tr_{|(0,1)\times \{1\}} \hat{U_\e}.
\end{equation}
Indeed we can check that, thanks to the symmetry assumption on \(a_0\), 
\[
\text{\( U_\v^{(1)}:(x_1,x_2)\mapsto\hat{U}_\e(\frac{1}{2}-x_1,x_2)\) and \( U_\v^{(2)}:(x_1,x_2)\mapsto \hat{U}_\e(x_1,\frac12 -x_2)\)}
\] satisfy the same equation as \(\hat{U}_\e\) in \(Q\)  with the same boundary condition. By the uniqueness 
result given in Corollary \ref{CorUniqSpecialSol} we obtain \(\hat{U}_\v^{(1)}=\hat{U}_\v^{(2)}=\hat{U}_\e\) and hence the equality of the traces on opposite faces of the square \(Q\). 

Now we set 
\begin{equation}
U_\e(x)= \hat{U}_\e(\tilde{x}_1,\tilde{x_2})
\end{equation}
if \(x \in G\) can be written as \(x=(k\delta +\tilde{x}_1 \delta,l \delta +\tilde{x}_2 \delta)\) for \( (\tilde{x}_1,\tilde{x}_2) \in Q\). Thanks to the homogeneous Neumann boundary condition satisfied by \(\hat{U}_\e\) on \(Q\) and because the traces of \(\hat{U}_\e\) are equal on opposite faces we can prove that \(U_\e\) satisfies 
\begin{equation}
\left\{
\begin{array}{rcll}
-\Delta U_\e &=&\frac{U_\e}{\e^2}(a_0(x/\delta)-U_\e^2) &\text{ in } G \\
 \p _\nu U_\e &=&0 &\text{ on } \p G.
\end{array}
\right.
\end{equation}
We can then apply the uniqueness result of Corollary \ref{CorUniqSpecialSol} to obtain that \(U_\e\) is the positive minimizer of \(E^\pin_\e\) in \(H^1(G)\). We then obtain that
\[ \|U_\e-1\|_{L^\infty(G)}=\|\hat{U}_\e-1 \|_{L^\infty(Q)}, \quad \|\nabla U_\e \|_{L^\infty(G)} =\delta^{-1} \| \nabla \hat{U}_\e\|_{L^\infty(Q)}. \]
The conclusion follows from the bound on the \(L^\infty\)-norm of  \(\hat{U}_\e\) and of its gradient which satisfies \eqref{eq:equation_sur_Q}. Note that the estimate on  \( \nabla \hat{U}_\e\) can be obtained as an interior estimate after extending \(\hat{U}_\e\) in a bigger square thanks to the homogeneous Neumann boundary condition.
\end{proof}

\subsection{The stationary ergodic case}

In this section we consider the case of a random stationary ergodic pinning term. More precisely we assume that \(a_\e\) is given by \eqref{def:a_e_random}. We will use the  Birkhoff ergodic theorem :

\begin{theorem}\label{th:Birkhoff}(see \cite[Th. 7.2]{Jikov_Kozlov_Oleunik_1994} and \cite{Dunford_Schwartz_1958})
Let \( (\O,\Sigma, \mu)\) be a probability space and \( (T=(T_x)_{x \in \R^d}\) be an action of \(\R^d\) on \(\O\) by measurable isomorphisms. Assume that \(a_1 \in L^p(\Omega)\) for some \(1\leq p < +\infty\). Then, for a.e.\ \(\omega \in \Omega\) the function \(a_1\left(T\left(\frac{\cdot}{\eta}\right)\omega  \right):\R^d \rightarrow \R\) weakly converges in \(L^p(\R^d)\) when \(\eta 
\rightarrow 0\). We denote by \( \N\left(a_1(T(x)\omega \right) \) its weak limit in \(L^p(\R^d)\). Then as a function of \(\omega\), \(\N\left(a_1(T(x)\omega) \right)\) is invariant by \(T\) and we have
\begin{equation}\nonumber
\int_\Omega \N\left(a_1(T(x) \omega )\right) d\mu =\mathbb{E}(a_1).
\end{equation}
Besides if \(T\) is ergodic then, \( \N\left(a_1(T(x)\omega) \right)=\mathbb{E}(a_1)\) for a.e.\ \( \omega \in \Omega \).
\end{theorem}
From this theorem we obtain, writing $\M:=\sqrt{\mathbb{E}(a_1)}$

\begin{theorem}\label{th:main3}
Let \(G\) be a  bounded \(\C^1\) domain of \(\R^d\). Let \(U_\e\) be the minimizer of \(E_\e^\pin\) in \(H^1(G)\) given by Definition \ref{Def.SpecialSol}, where \(a_\e\) is defined by \eqref{def:a_e_random}. Then
\begin{equation}\label{eq:conv_to_0bis}
\lim_{\e \to 0} \left\|U_\e-\M\right\|_{L^\infty(G)}=0 \text{ for a.e. } \omega \in \Omega.
\end{equation}
\end{theorem}

\begin{proof}
 Recall that, without loss of generality, we can assume that \( \mathbb{E}(a_1)=1\). By contradiction we assume that \eqref{eq:conv_to_0bis} is not true. Then there exists a set \(O \subset \Omega\) with \(\mu(O)>0\), such that for every \(\omega\in O\) there exists \(\eta^\omega >0\) and a sequence of points \((x^\omega_{\e})_{\e>0}\)=\((x_{\e})_{\e>0}\) such that \( |U_\e(x_\e,\omega)-1|\geq \eta\) for all \( \e>0\) small enough. We fix \(\omega\) and drop the subscript \(\omega\) in the following. We first assume that \(\rho_\e:=\dist( x_\e, \p G)\gg\e\). We then consider the blow-up function \( V_\e(y,\omega)= U_\e (x_\e+\e y,\omega)\) defined for \(y \in B(0,\rho_\e/\e) \subset G\). This function satisfies
\begin{equation}\label{eq:blow-up_eps2}
-\Delta V_\e = V_\e (b_\e-V_\e^2)  \text{ in }  B(0, \rho_\e/\e) 
\end{equation}
with \(b_\e(y):=a_\e(x_\e+\e y,\omega)=a_1\left(T\left(\frac{\e x+x_\e}{\delta}\right)\omega  \right) \). Besides, by the Lipschitz estimate \eqref{eq:gradient_estimate} we have that \(V_\e\) satisfies \( \|V_\e\|_{L^\infty(B(0, \rho_\e/\e) )}\leq C\) and \( \|\nabla V_\e \|_{L^\infty(B(0, \rho_\e/\e) )} \leq C\).  By the Arzela-Ascoli theorem we can extract a subsequence such that \(V_\e \rightarrow V_0\) locally uniformly in \(\R^d\). Since \(\e/\delta \rightarrow +\infty\), the strong oscillations of \(b_\e\) implies that (see Theorem \ref{th:Birkhoff})
\begin{equation}\nonumber
b_\e \rightharpoonup  \mathbb{E}(a_1)=1 \text{ in } L_{\text{loc}}^1(\R^d) \text{ for a.e.\ } \omega\in O.
\end{equation}
Thus we find that \( V_\e b_\e \rightharpoonup V_0 \) in \(\mathcal{D}'(\R^d)\) and \(V_\e^3 \rightarrow V_0^3\) locally uniformly  in \(\R^d\), almost surely. Passing to the limit in the sense of distributions in \eqref{eq:blow-up_eps2} we find that the limit \(V_0\) satisfies
\begin{equation*}
-\Delta V_0=V_0(1-V_0^2) \text{ in } \R^d.
\end{equation*}
since we have that \(m\leq V_0 \leq M\), by using Theorem 2.1 in \cite{Farina_2003} we conclude that \(V_0 \equiv 1\). Thus \(V_\e(0,\omega)=U_\e(x_\e,\omega) \rightarrow 1\) for a.e.\ \( \omega \in O\). This is a contradiction.

Now we assume that \( \dist(x_\e,\p G)=O (\v)\).  Thus we may define $y_\v:=\Pi_{\p G}(x_\v)$, the orthogonal projection of $x_\v$ on $\p G$. We then have $|x_\v-y_\v|=O(\v)$. Once again, up to passing to a further subsequence we may assume that $y_\v\to y_0\in\p G$.  We let  \( V_\e(y,\o):= U_\e (y_\e+\e y,\o)\) for $y\in B_\v=\dfrac{B(y_0-y_\v,1)\cap(G-y_\v)}{\v}$ and $\o\in O$.

 Then passing to the limit in \eqref{eq:blow-up_eps} and using the regularity of \(\p G\), we find that \(V_\v(\cdot,\o)\to V_0(\cdot,\o)\) locally uniformly in  \(\R^d_+\)  for $\o\in O$ and $V_0=V_0(\cdot,\o)$ satisfies
\begin{equation}\nonumber
\left\{
\begin{array}{rcll}
-\Delta V_0&=& V_0(1-V_0^2) & \text{ in } \R^d_+ \\
\p_\nu V_0 &=& 0 & \text{ on } \p \R^d_+.
\end{array}
\right.
\end{equation}
We can define a new function
\begin{equation*}
\tilde{V}_0(x)=\begin{cases} V_0(x) &\text{ if } x_d\geq 0 \\
V_0(x_1,\cdots,-x_d) & \text{ if } x_d<0
\end{cases}
\end{equation*} 
where \(x=(x_1,\cdots,x_d)\in \R^d\). We can check that \(\tilde{V}_0\) satisfies \(-\Delta \tilde{V}_0=\tilde{V}_0(1-\tilde{V}_0^2)\) in \(\R^d\) and we conclude as before that \(\tilde{V}_0 \equiv 1\).  This proves the theorem.
\end{proof}

Theorem \ref{th:main1} in the random case follows from Lemma \ref{lem:substitution} and Theorem \ref{th:main3}.

\begin{remark} As in Remark \ref{RmarkUnifConvMNeq1} we may adapt the proof to prove $U_\v\to\M$ in $L^\infty(G)$ for  a.e.  $\omega \in \Omega$ when $\M=\sqrt{\mathbb{E}(a_1)}\neq1$.
\end{remark}

\section{\(\Gamma\)-convergence and quasi-minimizers}

In this section we recall the definition of \(\Gamma\)-convergence of functionals and show that it allows us to describe the asymptotic behavior of quasi-minimizers of a family of functionals. 

\begin{definition}\label{def:Gamma_conv}
For \(\e \in (0,1],\) we consider a family of functionals
\begin{equation}\nonumber
F_\e: \I_\e \rightarrow (-\infty,+\infty], \quad \I_\e \text{ topological space }
\end{equation}
and 
\begin{equation}\nonumber
F :\I \rightarrow (-\infty,+\infty], \quad \I \text{ topological space }.
\end{equation}
We define
\begin{equation}\nonumber
\I_0:=\{ x \in \I\mid F(x)<+\infty \}.
\end{equation}
We say that \(F_\e\) \(\Gamma\)-converges to \(F\) as \(\e\rightarrow 0\) if for every \(\e\in (0,1]\) there exists \(P_\e:\I_\e \rightarrow \I\) such that:
\\ {\bf Lower bound:} If \(x\in \I_0\) and \(x_\e \in \I_\e\) is a sequence such that \(P_\v(x_\e) \rightarrow x \) (for the topology of \(\I\)) as \(\e \rightarrow 0\), then 
\begin{equation}\nonumber
\liminf_{\e \rightarrow 0} F_\e(x_\e) \geq F(x).
\end{equation}
\\{\bf Upper bound:} For every \(x \in \I_0\), for every $\v\in(0,1]$, there exists \(x_\e \in \I_\e\) such that \(P_\v(x_\e) \rightarrow x\) in \(\I\) and
\begin{equation}\nonumber
\limsup_{\e\rightarrow 0} F_\e(x_\e) \leq F(x).
\end{equation}
\end{definition}

The two first properties (Lower and Upper bounds) in the above definition are taken from section 3.1-\cite{Jerrard_Sternberg_2009} and are adapted from the original definition of De Giorgi. The adaptation comes from the fact that in Ginzburg-Landau theory the limiting space on which is defined the \(\Gamma\)-limit is not the same as the original space on which the Ginzburg-Landau functional is defined. 

In addition of these two properties  the supplementary compactness property is added.
\begin{equation}\label{CompactnessGammaProp}
\begin{cases}
\text{{\bf Compactness:} If, for some $\v_0\in(0,1]$, \(\sup_{\e \in (0,\v_0]} F_\e(x_\e) <+\infty\) then, for a  }\\\text{sequence $\v=\v_n\downarrow0$, there exist \(x\in \I_0\) and  a subsequence (still denoted by \(x_\e\))}\\\text{such that \(P_\v(x_\e) \rightarrow x\) in \(\I\).}
\end{cases}
\end{equation}

The notion of \(\Gamma\)-convergence has been conceived so that the {\it infima} of \(F_\e\) converge to the infimum of \(F\) and a family of minimizers of \(F_\e\) converges to a minimizer of  \(F\). This property remains true for a family of quasi-minimizers.

\begin{proposition}\label{prop:Gamma_conv_and_quasimin}
Let \(F_\e:\I_\e \rightarrow (-\infty,+\infty]\) be  a family of functionals defined on topological spaces \(\I_\e\) and \(F:\I\rightarrow (-\infty,+\infty]\) be a functional defined on a topological space \(\I\). Let us assume that \(F_\e\) \(\Gamma\)-converges towards \(F\) as \(\e \rightarrow 0\) and that the compactness property \eqref{CompactnessGammaProp} holds. Let \( (x_\e)_\e\) be a family of quasi-minimizers of \(F_\e\).

 If $F\not\equiv+\infty$ then there exists \(x \in \I\) such that,  up to a subsequence, \(P_\v(x_\e) \rightarrow x\) in \(\I\) and 
\begin{equation}\nonumber
F(x)=\inf_{y \in \I} F(y).
\end{equation}
\end{proposition}
In other words a family of quasi-minimizers also converges (up to a subsequence) towards a minimizer of the \(\Gamma\)-limit. The proof of this proposition is an adaptation of Theorem 1.21 in \cite{Braides_2002}.

Hence, thanks to Proposition \ref{prop:Gamma_conv_and_quasimin} and Theorem \ref{th:main2} we are able to understand the asymptotic behavior of minimizers of \(GL_\e^\pin\) thanks to existing \(\Gamma\)-convergence results on \(GL_\e\). This is the object of the following sections.

\section{Asymptotics for the pinned 2D Ginzburg-Landau energy}\label{sec:2D_results}

In this section we deduce from Theorem \ref{th:main1} results on the asymptotic behaviour of minimizers of \(GL_\e^\pin\) given by \eqref{eq:GL_pinned} with \(a_\e\) either given by \eqref{def:a_e_periodic} or by \eqref{def:a_e_random}. The main ingredient to pass from Theorem \ref{th:main1} to the description of minimizers of \(GL_\e^\pin\) is Proposition \ref{prop:Gamma_conv_and_quasimin}. In this section \(G\) is a smooth bounded domain of \(\R^2\).

We first introduce some notations. For \( (u,A)\in H^1(G,\mathbb{C}) \times H^1(G,\R^2)\) we define
\begin{equation}\label{eq:super_current_and_vorticity}
j(u)=(iu,\nabla_Au), \quad \mu(u,A)=\curl j(u)+ \curl A.
\end{equation}
Here \((iu,\nabla_Au)=\frac{i}{2} \left(u\overline{\nabla_Au}-\overline{u}\nabla_Au \right).\)
We let $\mathbb{M}(G)$ be the set of Radon measures. For \(\lambda>0\), we define \(E_\lambda : \mathbb{M}(G) \rightarrow (-\infty,+\infty]\) in the following way: for \(\mu \in \mathbb{M}(G) \cap H^{-1}(G)\) we consider \(h_\mu\) the solution of
\begin{equation}\label{def:h_mu}
\left\{
\begin{array}{rcll}
-\Delta h_\mu+h_\mu&=&\mu \text{ in } G \\
h_\mu &=& 1 \text{ on } \p G.
\end{array}
\right.
\end{equation}
We then set 
\begin{equation}\label{def:E_lambda}
E_\lambda (\mu)=\begin{cases}
\frac{\|\mu\|}{2\lambda}+\frac12 \int_G \left(|\nabla h_\mu|^2+|h_\mu-1|^2 \right) & \text{ if } \mu \in \mathbb{M}(G)\cap H^{-1}(G)\\
+\infty & \text{ otherwise }.
\end{cases}
\end{equation}

\begin{theorem}\label{th:obstacle_problem}
Assume that \(G \subset \R^2\) is a smooth simply connected bounded domain. Assume that \(\frac{h_{ex}}{|\log \e|}\rightarrow \lambda >0\) when \(\e\rightarrow 0\). We consider \( \{ (u_\e,A_\e)\}_\e\) a family of minimizers of \(G_\e^{\text{pin}}\). If we write \(u_\e=U_\e v_\e\) where \(U_\e\) is given by Definition \ref{Def.SpecialSol}. Then, as \(\e \rightarrow 0\),
\begin{equation}\label{eq:conv_vorticity}
\frac{\mu(v_\e,A_\e)}{h_{ex}} \rightarrow \mu_* \text{ in } (\C^{0,\gamma}(G))^* \text{ for every } \gamma \in (0,1),
\end{equation}
\begin{equation}\label{eq:conv_magnetic_field}
\frac{h_\e}{h_{ex}} \rightarrow h_{\mu_*} \text{ weakly in } H^1_1(G) \text{ and strongly in } W^{1,p}(G), \ \forall p<2,
\end{equation}
where \(\mu_*\) is the unique minimizer of \(E_\lambda\) given by \eqref{def:E_lambda}, and 
\begin{equation}\label{eq:conv_energy1}
\frac{GL_\e^\pin(u_\e,A_\e)-E_\e^\pin(U_\e)}{h_{ex}^2} \rightarrow E_\lambda(\mu_*).
\end{equation}
Moreover, 
\begin{equation}\label{eq:conv_density_energy}
\frac{g_\e(v_\e,A_\e)}{h_{ex}} \rightarrow \frac{1}{2\lambda}|\mu_*|+\frac12 \left( |\nabla h_{\mu_*}|^2+|h_{\mu_*}-1|^2 \right)
\end{equation}
and
\begin{equation}\label{eq:conv_concentration}
\left|\nabla \left( \frac{h_\e}{h_{ex}} \right) \right| \rightarrow \frac{1}{\lambda} \mu_*
\end{equation}
in the weak sense of measures.

Here $g_\e(u,A)=\frac{|\nabla u-iAu|^2}{2} +\frac{(1-|u|^2)^2}{4\e^2} +\frac{ |\curl A-h_{\text{ex}}|^2}{2}$. 
\end{theorem}

\begin{remark} We have \( j(u_\e)=U_\e^2 j(v_\e)\) and \( \mu(u_\e,A_\e)=U_\e^2\mu(v_\e,A)+2U_\e \nabla^\perp U_\e \cdot j(v_\e)+\curl A (1-U_\e^2)\). 
\end{remark}

\begin{proof}
We use Theorem \ref{th:main1}, Proposition \ref{prop:Gamma_conv_and_quasimin} and the \(\Gamma\)-convergence result on \(GL_\e/h_{ex}^2\) in this regime of applied magnetic field, cf.\ Theorem 7.1 in \cite{Sandier_Serfaty_2007} to deduce \eqref{eq:conv_vorticity}, \eqref{eq:conv_magnetic_field} and \eqref{eq:conv_energy1}. Note that in Theorem 7.1 in \cite{Sandier_Serfaty_2007} the \(\Gamma\)-convergence result is obtain with
\[ \I_\e:= H^1(G,\mathbb{C})\times H^1(G,\R^2), \quad \I:= \left( \mathcal{C}^{0,\gamma}(G) \right)^*\times L^2(G,\R^2)\]
for any \(\gamma \in (0,1)\) where \(\I\) is endowed with the product topology,  \(\left( \mathcal{C}^{0,\gamma}(G) \right)^*\) is endowed with the weak-\(*\) topology and \( L^2(G,\R^2)\) with the weak topology. Furthermore with the notations of definition \ref{def:Gamma_conv} we have
\[\begin{array}{rcll}
P_\e: & H^1(G,\mathbb{C})\times H^1(G,\R^2)& \rightarrow &\left( \mathcal{C}^{0,\gamma}(G) \right)^*\times L^2(G,\R^2) \\
 & (u_\e,A_\e) & \mapsto & \left( \mu(u_\e,A_\e), \curl A_\e \right).
\end{array}
\]
Statements \eqref{eq:conv_density_energy} and \eqref{eq:conv_concentration} follow exactly as in the proof of Theorem 7.2 in \cite{Sandier_Serfaty_2007}.
\end{proof}

\begin{theorem}
Assume that \( |\log \e| \ll h_{ex}\ll1/\e^2\) as \(\e \rightarrow 0\). Let \( \{(u_\e,A_\e)\}_\e\) be a family of minimizers of \(GL_\e^{\text{pin}}\) in \(\mathcal{H}\). We set \(u_\e=U_\e v_\e\) where \(U_\e\) is given by Definition \ref{Def.SpecialSol}. Then
\begin{equation}\nonumber
\frac{2g_\e(v_\e,A_\e)}{h_{ex}}|\log \e \sqrt{h_{ex}|} \rightharpoonup {\rm d}x \quad \e \rightarrow 0
\end{equation}
in the weak sense of measures and
\begin{equation}\nonumber
\min_{(u,A)\in \H} G_\e(u,A) \simeq \frac{|G|}{2}h_{ex}|\log \e \sqrt{h_{ex}}| \quad \text{ as } \e \rightarrow 0.
\end{equation}
Besides
\begin{equation}\nonumber
\frac{h_\e}{h_{ex}} \rightarrow 1 \text{ in } H^1(G) \text{ and } \quad \frac{\mu(v_\e,A_\e)}{h_{ex}} \rightarrow {\rm d}x \text{ in } H^{-1}(G).
\end{equation}
\end{theorem}

\begin{proof}
This follows from Theorem 8.1 and Corollary 8.1 in \cite{Sandier_Serfaty_2007} along with Theorem \ref{th:main1} and Proposition \ref{prop:Gamma_conv_and_quasimin}. 
\end{proof}

Unfortunately Theorem \ref{th:main1} is not sufficient to describe the behavior of minimizers of \(G_\e ^{\text{pin}}\) near the so-called first critical field or more generally when there is a number of vortices much smaller than the applied magnetic field \(h_{ex}\). This is because the leading order term in the asymptotic expansion of \(GL_\e(v_\e,A_\e)\) is independent of the position of the vortices. In the so-called intermediate regime it is also independent of the number of vortices and is of order \(h_{ex}\). However, with an explicit rate of convergence of \(U_\e\), the positive minimizer of \(E_\e^\pin\) in \(H^1(G)\), we can give condition on this rate such that results of chapters 9-10-11 in \cite{Sandier_Serfaty_2007} can be applied to describe the asymptotic behavior of minimizers near the first critical field.

We first introduce some notations: We define \(h_0\) to be the solution of 
\begin{equation}\nonumber
\left\{
\begin{array}{rcll}
-\Delta h_0+h_0 &=&0 & \text{ in } G \\
h_0&=& 1 & \text{ on } \p G
\end{array}
\right.
\end{equation}
and
\begin{equation}\nonumber
\xi_0:=h_0-1 \quad \text{ and } \underline{\xi_0}=\min_G \xi_0.
\end{equation}
We suppose that \(\xi_0\) has a unique minimizer \(p\) in \(G\). We set 
\begin{equation}\nonumber
Q(x):= D^2(\xi_0)(p)(x,x)
\end{equation}
and we assume that \(Q\) is a positive definite quadratic form.
We set 
\begin{equation}\nonumber
J_0=\frac12 \int_G |\nabla h_0|^2+|h_0-1|^2=\frac12 \|\xi_0\|^2_{H^1(G)}.
\end{equation}
We also set 
\begin{equation}\nonumber
H^0_{c_1}:=\frac{1}{2|\underline{\xi_0}|}|\log \e|.
\end{equation}
We denote by \(\mathcal{G}\) the modified Green function, solution to 
\begin{equation}\nonumber
\left\{
\begin{array}{rcll}
-\Delta_x \mathcal{G}(x,y)+\mathcal{G}(x,y) &=& \delta_y & \text{ in } G \\
\mathcal{G}(x,y) &=&0 & \text{ on } \p G,
\end{array}
\right.
\end{equation}
and we set
\begin{equation}\nonumber
S_G(x,y)=2\pi \mathcal{G}(x,y)+\log |x-y|.
\end{equation}
For \(n\in \mathbb{N}\) we set \( \ell:= \sqrt{\frac{n}{h_{ex}}} \). We denote by \(\varphi\) the blow-up centered at \(p\) for the scale \(\ell\) defined by
\begin{equation}\nonumber
\varphi(x)=\frac{x-p}{\ell}.
\end{equation}
If \(\mu \) is a measure we will denote by \( \tilde{\mu}\) its push-forward by the mapping \(\varphi\), {\it i.e.} \(\tilde{\mu}(U)=\mu\left( \varphi^{-1}(U)\right)\) for every \(U\) measurable subset of \(\R^2\). If \(x\) is a point then we let \( \tilde{x}=\varphi(x)\). Now, we define a functional on the space of probability measures on \(\R^2\) denoted by \(\mathcal{P}\):
\begin{equation}\nonumber
I(\mu)= -\pi \int_{\R^2} \int_{\R^2} \log|x-y| d\mu(x)d\mu(y)+\pi \int_{\R^2} Q(x)d\mu(x) \quad \text{ for } \mu \in \mathcal{P}.
\end{equation}
It is known that the infimum \(\inf_{\mu \in \mathcal{P}} I(\mu)\) is uniquely achieved (see e.g.\ \cite{Saff_Totik_1997}).
We denote by \(\mu_0\) the minimizer and we let
\begin{equation}\nonumber
I_0:=I(\mu_0)=\inf_{\mu \in \mathcal{P}} I(\mu).
\end{equation}
For \(n\in \mathbb{N}\) we define
\begin{align}
g_\e(n)&:=h_{ex}^2J_0+\pi n|\log \e|-2\pi n h_{ex}|\underline{\xi_0}|+\pi(n^2-n)\log \frac{1}{\ell}\\
& \phantom{aaaaaa} +\pi n^2 S_G(p,p)+n^2I_0.
\end{align}
We recall from Lemma 9.5 in \cite{Sandier_Serfaty_2007}:

\begin{lemma}
There exist constant \(\alpha,\e_0>0\) and for each \(0<\e<\e_0\) an increasing sequence \( (H_n)_n\) defined for integers \(0\leq n \leq \alpha |\log \e| \), such that if \(h_{ex}>H^0_{c_1}/2\), then \(n\) minimizes \(g_\e\) over the integers in the interval \( [0,\alpha |\log \e|]\) if and only if
\[ h_{ex}\in [H_n,H_{n+1}].\]
\end{lemma}

We can now state
\begin{theorem}
Assume that \(h_{ex}\) is such that 
\[ |\log |\log \e|| \ll h_{ex}(\e)-H^0_{c_1} \ll|\log \e|, \]
let \(N_\e\) be a corresponding minimizer of \(g_\e(\cdot)\) over \([0, \alpha |\log \e|]\).
Let \( (u_\e,A_\e)\) be a minimizer of \(GL_\e^\pin\), we write \(u_\e=U_\e v_\e\) where \(U_\e\) is given by Definition \ref{Def.SpecialSol}. Assume that 
\begin{equation}\label{hyp:rate_conv}
\|U_\e-1\|_{L^\infty(G)}\times g_\e(N_\e)=o_\e(N_\e^2).
\end{equation} 
Then for any \(\gamma \in (0,1)\)
\begin{equation}\nonumber
\frac{\tilde{\mu}(u_\e,A_\e)}{2\pi N_\e} \rightharpoonup \mu_0 \text{ in } \left(\mathcal{C}_c^{0,\gamma}(\R^2) \right)^*,
\end{equation}
where \(\mu_0\) is the unique minimizer of \(I\) and 
\begin{equation}\nonumber
GL_\e(v_\e,A_\e)=g_\e(N_\e)+o_\e(N_\e^2), \quad GL_\e^\pin(u_\e,A_\e)=E_\e^\pin(U_\e)+g_\e(N_\e)+o_\e(N_\e^2).
\end{equation}
\end{theorem}

\begin{proof}
Again we deduce this theorem from Theorem \ref{th:main1}, Proposition \ref{prop:Gamma_conv_and_quasimin} and existing results in the literature. Here the results used are Theorem 9.1-9.2 in \cite{Sandier_Serfaty_2007}.
The assumption \eqref{hyp:rate_conv} is used to guarantee that \(g_\e(N_\e) \times \|U_\e-1\|_{L^\infty(G)}\) is negligible compared to all the terms of \(g_\e\). 
\end{proof}

\textbf{Remark:} From Proposition \ref{prop:speed_convergence}, assumption \eqref{hyp:rate_conv} is satisfied for example when \( \delta =O_\e(\e^2)\) and \(\frac{\delta^{1/4}}{\e^{1/2}}h_{ex}^2=o_\e(1)\). This means \(\delta =o_\e\left(\frac{\e^2}{h_{ex}^{8}} \right)\).

\medskip

Finally it remains to examine the case of a bounded number of vortices. We let
\begin{equation}\label{DefFEps}
f_\e(n)=h_{ex}^2J_0+\pi n \log \frac{\ell}{\e}-2\pi n h_{ex}|\underline{\xi_0}|+\pi n^2 S_G(p,p)+\pi n^2 \log \frac{1}{\ell}.
\end{equation}

We recall from Lemma 12.1 in \cite{Sandier_Serfaty_2007}

\begin{lemma}
For every \(\e>0\), there exists an increasing sequence \( (H_n(\e))_n\), \(H_0=0\), such that the following holds. Given \(n\geq 0\) independent of \(\e\), if \(h_{ex}(\e)\gg1\) is such that 
\[g_\e(n) \leq \min \left( g_\e(n-1),g_\e(n+1) \right)+o_\e(1), \]
then 
\[H_n-o_\v(1)\leq h_{ex} \leq H_{n+1}+o_\e(1).\]
Moreover the following asymptotic expansion holds as \(\e \rightarrow 0\)
\[H_n=\frac{1}{2|\underline{\xi_0}|}\left[ |\log \e|+(n-1)\log \frac{|\log \e|}{2|\underline{\xi_0}|}+K_n \right]+o_\e(1) \]
where
\begin{multline*}
K_n=(n-1)\log \frac{1}{n}+\frac{n^2-3n+2}{2}\log \frac{n-1}{n} \\
+\frac{1}{\pi}\left( \min_{(\R^2)^n } w_n- \min_{(\R^2)^{n-1}} w_{n-1} +\gamma +(2n-1)\pi S_{G}(p,p)\right).
\end{multline*}
Here \(\gamma\) is a universal constant and 
\begin{equation}\label{def:renormalized_energy}
w_n(x_1,\cdots,x_n)=-\pi \sum_{i \neq j} \log |x_i-x_j|+\pi n \sum_{i=1}^n Q(x_i).
\end{equation}
\end{lemma}

We can now state
\begin{theorem}
Assume that \(N\in \mathbb{N}\). There exists \(c_\e \rightarrow 0\) as \(\e \rightarrow 0\) such that if \(\e< \e_0(N)\) and 
\begin{equation}\nonumber
H_N+c_\e \leq h_{ex} \leq H_{N+1}-c_\e,
\end{equation}
and if \((u_\e,A_\e)\) is a minimizer of \(GL_\e^\pin\), then writing \(u_\e=U_\e v_\e\) where \(U_\e\) is the minimizer of \(E_\e^\pin\) given by Definition \ref{Def.SpecialSol} then if 
\begin{equation}\label{hyp:hyp_2}
\|U_\e-1\|_{L^\infty(G)} \times f_\e(N)=o_\e(1)
\end{equation}
then \(v_\e\) has \(N\) vortices \(a_1^\e,\cdots,a_N^\e\) and, possibly after extraction and letting \(\tilde{a}_i^\e:=(a_i^\e-p)/\ell\), the \(N\)-tuple \( (\tilde{a}_1^\e,\cdots,\tilde{a}_N^\e)\) converges as \(\e \rightarrow 0\) to a minimizer of \(w_N\) given by \eqref{def:renormalized_energy} and 
\begin{equation}\nonumber
GL_\e(v_\e,A_\e) =f_\e(N)+\min_{(\R^2)^N}w_N+N\gamma +o_\e(1) \text{ as } \e \rightarrow 0.
\end{equation}
\end{theorem}

\textbf{Remark:} Assumption \eqref{hyp:hyp_2} is satisfied for example when \(\delta=O_\e(\e^2)\) and 
\[ \frac{\delta^{1/4}}{\e^{1/2}}\times h_{ex}^2=o_\e(1) \]
leading to \(\delta =o_\e\left(\frac{\e^2}{h_{ex}^8} \right).\)

\begin{proof}
Here we use Theorem \ref{th:main1}, Proposition \ref{prop:Gamma_conv_and_quasimin} and Theorems 12.1 in \cite{Sandier_Serfaty_2007}. The hypothesis \eqref{hyp:hyp_2} is here to guarantee that \( GL_\e(v_\e,A_\e)\times \|U_\e-1\|_{L^\infty(G)}\) is much smaller than all the terms in the asymptotic expansion of \(\inf_{(v_\e,A_\e)} GL_\e(v_\e,A_\e)\).
\end{proof}

\section{Asymptotics for the pinned 3D Ginzburg-Landau energy}\label{sec:3D_results}

Let \(G \subset \R^3\) be a smooth bounded domain. In this section we consider a 3D-variant of the  energy \eqref{eq:GL_pinned}. Here we use differential forms formalism. We define 
\begin{equation}
\F^\pin_\e(u,A)=\frac12 \int_G |du-iAu|^2+\frac{1}{4\e^2}\int_G (a_\e(x)-|u|^2)^2+\frac12 \int_{\R^3}  |dA-h_{ex}|^2,
\end{equation}
here \(u \in H^1(G,\mathbb{C})\), \( du\) is a \(1\)-form, \(h_{ex} \in L^2_{\text{loc}}(\Lambda^2 \R^3)\) is a \(2\)-form and \(A \in H^1(\Lambda^1 \R^3)\) is a \(1\)-form and \(a_\e\) is defined by \eqref{def:a_e_periodic} or by \eqref{def:a_e_random}.
We define 
\begin{equation}
\dot{H}^1_*(\Lambda^1 \R^3)=\{ A\in \dot{H}^1(\Lambda^1\R^3)\mid \ d^*A=0\}
\end{equation}
and we endow this space with the inner product
\begin{equation}
(A,B)_{\dot{H}^1_*(\Lambda^1\R^3)}:=(dA,dB)_{L^2(\Lambda^2 \R^3)}
\end{equation}
for which \(\dot{H}^1_*(\Lambda^1\R^3)\) is a Hilbert space.
For \(u \in H^1(G,\mathbb{C})\) we define (here we denote $u=u^1+i u^2$, $u^1,u^2\in H^1(G,\mathbb{R})$)
\begin{equation}
ju:=(iu,du)=u^1du^2-u^2du^1, \quad Ju=du^1\wedge du^2=\frac12 d (ju).
\end{equation}

\begin{theorem}
Assume that \(h_{ex}=d A_{ex,\e}\) and that there exists \(A_{ex,0} \in H^1_{loc}(\Lambda^1\R^3)\) such that
\begin{equation}\nonumber
\frac{A_{ex,\e}}{|\log \e|}-A_{ex,0} \rightarrow 0 \text{ in } \dot{H}^1_*(\Lambda^1\R^3).
\end{equation}
Let \( (u_\e,A_\e)\in H^1(G,\mathbb{C})\times \left[ A_{ex,0}+\dot{H}^1_*(\Lambda^1\R^3)\right]\) be  a family of minimizers of \(\mathcal{F}_\e^\pin\). We write \(u_\e=U_\e v_\e\) where \(U_\e\) is the minimizer of \(E_\e^\pin\) given by Definition \ref{Def.SpecialSol}. Then, up to a subsequence  we have
\begin{equation}\nonumber
\frac{A_\e}{|\log \e|} \rightharpoonup A_* \text{ in } \dot{H}^1_*(\Lambda^1\R^3)
\end{equation}
for some \(A_*\in A_{ex,0} +\dot{H}^1_*(\Lambda^1\R^3)\),
\begin{equation}\nonumber
\frac{j v_\e}{|\log \e|}\rightharpoonup w_* \text{ in } L^\frac{8}{6}(\Lambda^1 G)
\end{equation}
\begin{equation}\nonumber
\frac{jv_\e}{|v_\e| |\log \e|} \rightharpoonup w_* \text{ in } L^2(\Lambda^1G),
\end{equation}
\begin{equation}\nonumber
\frac{Jv_\e}{|\log \e|}=\frac{d (jv_\e)}{2|\log \e|}\rightarrow J_* \text{ in } W^{-1,p}(\Lambda^2G) \ \forall p<3/2
\end{equation}
for some \((J_*,w_*)\in \mathcal{A}_0:= \{ (J,w)\mid J \text{ is an exact measure-valued } 2-\text{ form in }G, \ v \in L^2(\Lambda^1G)\}\) and \(J_*=\frac{dw_*}{2}\in H^{-1}(\Lambda^2G)\).
Besides \( (w_*,A_*)\) is a minimizer of the following functional defined for \( (v,A) \in L^2(\Lambda^1G)\times \left[  A_{ex,0}+\dot{H}^1_*(\Lambda^1\R^3) \right] \) by
\begin{equation}\nonumber
\mathcal{F}(v,A)=\begin{cases}
\frac12 \|dv\|+\frac12 \|v-A\|^2_{L^2(\Lambda^1G)}+\frac12 \|dA-dA_{ex,0}\|^2_{L^2(\Lambda^2 \R^3)} \\
\phantom{aaaaaaaaaaaaaaaaaaaaaaa} \text{ if } \|dv\|=|dv|(\O)<+\infty \\
+\infty  \text{ otherwise }.
\end{cases}
\end{equation}
\end{theorem}

\begin{proof}
It is easy to check that an analogue of Lemma \ref{lem:substitution} holds for the \(3d\)-magnetic Ginzburg-Landau energy. With Proposition \ref{th:main2} and \ref{th:main3} we find that the analogue of Theorem \ref{th:main1} is true for the \(3d\)-Ginzburg-Landau energy. We conclude by using, Proposition \ref{prop:Gamma_conv_and_quasimin} and Theorem 4 in \cite{Baldo_Jerrard_Orlandi_Soner_2012}.
\end{proof}
\section{Asymptotics for the pinned Allen-Cahn energy}\label{sec:Allen-Cahn_results}
In this section $G$ is a \(\C^1\) bounded open set of $\R^d$, \(d\geq 1\). By taking \(A=0\) and \(h_{ex}=0\) we are able to describe  the asymptotic behavior of a pinned Allen-Cahn functional. For \(u \in H^1(G,\R)\) we define
\begin{equation}\label{def:pinned_Allen_Cahn}
AC_\e^{\text{pin}}(u)=\e \int_G |\nabla u|^2+\frac{1}{\e}\int_G (a_\e(x)-u^2)^2
\end{equation}
where \(a_\e\) is given by \eqref{def:a_e_periodic} or \eqref{def:a_e_random}.
\begin{theorem}
Let \( 0< \beta <1\) and \((u_\e)_\e \subset H^1(G,\R)\) be a family of minimizers of the pinned Allen-Cahn energy \eqref{def:pinned_Allen_Cahn} under the constraint \( \frac{1}{|G|}\int_G u_\e=\beta\). Then we can write \(u_\e=U_\e v_\e\) with \(U_\e\) given by Definition \ref{Def.SpecialSol} and we have that there exists \(v \in BV(G,\{\pm 1 \})\) such that
\begin{equation}\nonumber
v_\e \rightarrow v \text{ in } L^1(G) 
\end{equation}
and \(v\) minimizes
\begin{equation}\nonumber
A(w)= \frac{4}{3}\int_G \abs{Dw}
\end{equation}
for \(w\in BV(G,\{\pm 1 \})\) under the constraint \(\frac{1}{|G|}\int_G w=\beta\).
\end{theorem}

\begin{remark} Recall that we normalized the average of \(a_0\) and \(a_1\) such that these quantities are equal to \(1\).
\end{remark}
\begin{proof}
This follows from an analogue of Theorem \ref{th:main1} which is true thanks to Lemma \ref{lem:substitution} and Theorem \ref{th:main2} and Theorem \ref{th:main3}. We conclude with Proposition \ref{prop:Gamma_conv_and_quasimin} and the \(\Gamma\)-convergence results in \cite{Modica_1987}
\end{proof}
%
%

\smallskip

\textbf{Acknowledgements:} We would like to warmly thank Alberto Farina for giving us the reference \cite{Farina_2003}.
\bibliographystyle{abbrv}
\bibliography{biblio}

\end{document}